\documentclass{amsart}[11pt]
\usepackage{graphicx}
\usepackage{stmaryrd}
\usepackage{enumitem}
\usepackage{color}
\numberwithin{equation}{section}
\usepackage{xcolor}

\DeclareMathAlphabet\mathbfcal{OMS}{cmsy}{b}{n}
    \def\a{\alpha} \def\b{\beta}  \def\({\left (} \def\){\right )}
\def\<{\left\langle} \def\>{\right\rangle}

\newtheorem{thm}{Theorem}[section]
 \newtheorem{lem}[thm]{Lemma}  
\newtheorem{acknowledgement}{Acknowledgement}

\newcommand{\R}{\mathbb R}

\newtheorem{theorem}{Theorem}[]
\newtheorem{Lemma1}{Lemma}[section]    
   
\newtheorem{Proposition}[Lemma1]{Proposition}

\theoremstyle{definition}
\newtheorem{definition}{Definition}[section]

\begin{document}

\title{Finite time blowup of the   $n$-harmonic flow on $n$-manifolds}

\author {Leslie hon-nam Cheung and Min-Chun Hong}

\address{Department of Mathematics, The University of Queensland\\ Brisbane, QLD 4072, Australia}  \email{l.cheung2@uq.edu.au; hong@maths.uq.edu.au}

\begin{abstract}
In this paper, we generalize  the no-neck result of Qing-Tian \cite{QT} to show that there is no neck during blowing up for the $n$-harmonic flow as $t\to\infty$. As an application of the no-neck result, we settle a conjecture of Hungerb\"uhler \cite {Hung} by constructing an example to show that the $n$-harmonic map flow  on  an $n$-dimensional Riemannian manifold blows up  in finite time for $n\geq 3$.
\end{abstract}

\markright {$n$-harmonic flow}

\maketitle

\section{Introduction}

Let $M$ be an $n$-dimensional Riemannian manifold without
boundary, and let $N$ be another
$m$-dimensional compact Riemannian manifold without boundary (isometrically embedded into $\R^L$).  In local coordinates,   a smooth Riemannian metric $g$ of $M$ can be
represented by
\[g = g_{ij}  dx_i\otimes dx_j,\] where $(g_{ij})$ is a positive
definitive symmetric $n\times n$ matrix.
 The volume element of $(M; g)$ is defined by
 \[dv= \sqrt {|g|} dx \quad \mbox{with } |g|=\mbox{det }(g_{ij}).\]
 For a map $u: M\to N\subset \R^L$,  the $n$-energy functional of $u$ is defined by
\[E_n(u; M)=\frac 1 n\int_M  |\nabla u |^n\,dv,\]
where $|\nabla u| $  is the  gradient norm  given by
\[  |\nabla u (x)|^2 = \sum_{\alpha, i,j}g^{ij}(x) \frac {\partial u^{\alpha}}{\partial x_i}\frac {\partial u^{\alpha}}{\partial x_j}\]
with $(g^{ij})=(g_{ij})^{-1}$ the inverse matrix of $(g_{ij})$.
A $C^1$-map $u$ from $M$ to $N$ is said to be an  $n$-harmonic map  if
$u$ is a critical point of the $n$-energy functional; i.e. it satisfies
\begin{equation}\label{har} \frac{1}{\sqrt{|g|}}\frac{\partial}{\partial x_{i}}\left[  |\nabla
u|^{n-2}g^{ij}\sqrt{|g|}\frac{\partial}{\partial x_{j}}u\right]
+|\nabla
u|^{n-2}A(u)(\nabla u,\nabla u)=0\quad  \mbox { in } M,\end{equation}
where $A$ is the second fundamental form of $N$.

In 1964, Eells and Sampson \cite{ES} investigated the existence problem of harmonic maps in a homotopic class; i.e.
``Given a smooth map $u_0:M\to N$, is there a harmonic map $u$, which is homotopic to $u_0$?'' (See \cite{EL}).

For the target manifold $N$ with non-positive sectional curvature, Eells and Sampson \cite{ES} proved the first existence result of  harmonic maps in a homotopic class by introducing the ``heat flow method''. The heat flow method transforms the existence  problem to an evolution problem. Since then, questions on existence and regularity of harmonic maps and their flows have been attracted a great attention (See \cite{EL}). One of the key components of the heat flow method for answering the Eells-Sampson question is to prove existence of a global solution  to the harmonic map flow. In 1975, Hamilton  \cite {Ha} proved local existence of the heat flow of harmonic map; i.e. the solutions of the heat flow of harmonic map exists locally. If the solution exists only in a finite   interval $[0, T_{max})$ with $T_{max}<\infty$ and cannot be extended any further, then we say that the solution blows up in finite time  $T_{max}$. In the two dimensional case (i.e. $n=2$), Struwe \cite{St} proved global existence of a unique  weak solution  to the harmonic map flow, where the  solution is smooth except for a finite set of point singularities.  In 1989, Coron-Ghidaglia \cite{CG} constructed the first example to show that for $n\geq 3$, the harmonic map heat flow from $S^n$ into $S^n$  blows up in finite time.
However,  when $n=2$,  the Dirichlet energy $E_2$ on the $2$-dimensional manifold is conformally invariant on its critical dimension. In addition, H\'elein \cite{H} proved that any weak harmonic map from surfaces is  smooth. Thus, it was widely believed during the time that the harmonic map heat flow would not blow up in finite time on the $2$-dimension manifold. In 1992, Chang, Ding and Ye \cite{CDY} made a breakthrough by constructing a counter-example that harmonic map heat flow on $S^2$  can blow up in finite time.

In higher dimensions (i.e. $n>2$), $E_n$ is also conformally invariant on the $n$-dimensional manifold $M$. Motivated by the Eells-Sampson question on harmonic maps, one can ask whether a given map from an $n$-dimensional manifold to another manifold
can be deformed into an  $n$-harmonic map. Related to this question, Hungerb\"uhler \cite {Hung} studied the $n$-harmonic flow in the following setting:
\begin{equation}\label{n-flow} \frac {\partial u}{\partial t}=\frac{1}{\sqrt{|g|}}\frac{\partial}{\partial x_{i}}\left[  |\nabla
u|^{n-2}g^{ij}\sqrt{|g|}\frac{\partial}{\partial x_{j}}u\right]
+|\nabla
u|^{n-2}A(u)(\nabla u,\nabla u)\end{equation}  and  generalized the result of Struwe \cite{St} by
 proving that there
exists a global weak solution $u: M\times [0,+\infty)\to N$ of (\ref{n-flow}) with initial value $u_0$  such that $u\in
{\bf C}^{1,\alpha }(M\times (0,+\infty )\backslash
\{\Sigma_k\times T_k \}_{k=1}^L)$ for a finite
number of times $\{T_k\}_{k=1}^L$ and a finite number of singular
closed sets $\Sigma_k \subset M$ for $k=1,..., L$  with an integer
$L$, depending only $M$ and $u_0$. However,  it is still  unknown whether the singular set $\Sigma_k$ of the flow (\ref{n-flow}) at the singular $T_k$ is finite. In order to sort out this issue, the second author \cite{Hong} introduced a  rectified $n$-harmonic map flow from an $n$-dimensional from $M$ to $N$ and proved existence of a global
solution, which is regular except for a finite number of points,  of the rectified $n$-harmonic map flow.

Based on the fundamental result of Chang-Ding-Ye \cite{CDY} for $n=2$, it is an interesting question whether the $n$-harmonic flow (\ref{n-flow}) blows up in finite time for $n\geq 3$. Supported by some numerical evidence, Hungerb\"uhler   (\cite{Hungerb1994}, \cite {Hung}) conjectured the phenomenon of finite time blow-up of the $n$-harmonic flow  for $n\geq 3$. Later, Chen, Cheung, Choi and Law \cite{CCCL} followed the method of Chang-Ding-Ye to construct  an example that  the $n$-harmonic flow (\ref{n-flow}) blows up in finite time  for $n=3$. However, due to the nonlinearity and degeneracy of the $n$-harmonic maps, they \cite{CCCL}   also pointed out that their proofs could not be applied to the cases when $n > 3$. Therefore, the  conjecture of Hungerb\"uhler for $n >3$ has remained open since then.

On the other hand, Qing-Tian \cite {QT}  suggested a program to prove the finite time blow-up of the harmonic map flow for $n=2$ through an application of the no-neck result for the harmonic map flow as $t \to \infty$ and  constructing  a special target manifold $N$ with a proper topology.  Recently, Chen-Li \cite {CL}  verified  the Qing-Tian  program by constructing a special target manifold $N$ with a proper topology to show that  the harmonic map flow blows up at finite time for $n=2$.  Later, Liu and Yin \cite{LY} successfully applied this idea to  construct a proper manifold $N$ to show that the bi-harmonic maps flow on $4$-manifolds blows up  at finite time.

In this paper,  we apply the Qing-Tian  program  to   confirm the conjecture of Hungerb\"uhler on the $n$-harmonic map flow. Firstly, we define:

\begin{definition}
$u$ is said to be a regular solution to the $n$-harmonic map flow \eqref{n-flow} in $M\times (0,T]$ if $u\in C^{0}(M \times (0,T] ; N)$ with $T\leq \infty$ is a solution of \eqref{n-flow} satisfying
  \[\int_0^T\int_M \left|\nabla(|\nabla u|^{\frac{n-2}{2}} \nabla u)\right |^2 +|\nabla u|^{2n}\,dv\,dt\leq C (T).\]
\end{definition}

We generalize the no-neck result of Qing-Tian \cite {QT} to the $n$-harmonic map flow as follow:

\begin{theorem}\label{no_neck} Let $u$ be a  regular solution to the flow (\ref {n-flow}) in $M\times [0,\infty)$ with initial value $u_0\in C^1(M, N)$. For a sequence $t_i\to \infty$,
there is a sub-sequence, still denoted by $t_i$, such that as $t_i\to \infty $, $u(x, t_i)$ converges to an $n$-harmonic map $u_{\infty}$  in $C_{loc}^{1,\a} (M\backslash \{x^1, \cdots, x^L\}, N)$ for some positive $\a<1$, where $u_{\infty}$ can be extended  to $C^{1,\a}(M,N)$.
Moreover, we have

\textbf{(i)} (Energy identity) There are a finite number of
$n$-harmonic maps $\omega_{k,l}$ (also called bubbles) on $S^{n}$  for $k=1,\cdots, L$ and $l=1,\cdots, J_k$  such that

\[\lim_{t_i \nearrow \infty }  E_n(u(\cdot ,t_i); M)= E_n(u_{\infty}; M)+\sum_{k=1}^{L}\sum_{l=1}^{J_k}   E_n  (\omega_{k,l}  ; S^{n}). \]

\textbf{(ii)} (No-neck result) There is no neck between the limiting map and bubbles; \\i.e. the image  $$u_{\infty}(M)\cup \bigcup_{k,l}\omega_{k,l}(S^n)$$ is a connected set.
\end{theorem}

One of the fundamental rules for bubble blowing is  the bubble-neck decomposition. During the bubbling procedure, the energy identity implies that the energy is conserved. This means that the loss of energy under the limiting process can be recovered by the energy of a finite number of bubbles. Readers can refer to the pioneering work on the energy identity by Jost \cite{J}, Parker \cite{P} regarding the harmonic maps from surfaces and by Ding-Tian \cite{DT} for the harmonic map flow. For $n$-harmonic maps with $n\geq 3$, the isolated singularities are removable due to Duzaar-Fuchs \cite{DF} and the energy identity was provided by Wang-Wei \cite{WW} for a sequence of approximate $n$-harmonic maps. In particular, one can use the standard blow-up argument as in Ding-Tian \cite{DT}  to reduce the multiple  bubble problem to the single bubble case. See more details for the bubble-neck decomposition  of $n$-harmonic maps in \cite{Hong}. These results allow us to construct the bubbling argument in the setting of the $n$-harmonic maps.

In order to provide an example to show that the $n$-harmonic flow can blow up in finite time, the key step is to generalize the no-neck result of Qing-Tian \cite{QT}.  However, those no-neck results in \cite{LZ} and \cite{QT} heavily rely on a key estimate in Ding-Tian's work (Lemma 2.1, \cite{DT}) which only works for the case of harmonic maps. To settle this open problem, we generalize the Ding-Tian estimate to the context of $n$-harmonic maps (Lemma \ref{Osc}) and then apply it to prove the no-neck property for the $n$-harmonic map flow.

Secondly, we apply Theorem  \ref{no_neck} to prove the main result of this paper:
\begin{theorem}\label{main_theorem} Let $X$ be any closed manifold of dimension $m>n$ with
nontrivial $\pi_n(X)$, and let $N=  X\,\#T^m$ be the connected sum of $X$  with the torus
$T^m$. Then there are infinitely many initial maps $u_0 : S^n \to N$ such that the $n$-harmonic map flow \eqref{n-flow} with initial value $u_0$ blows up in finite time.
\end{theorem}

Besides the finite time blow-up result on the harmonic map flow by Chen-Li \cite{CL}, another related  evolution problem   to the $n$-harmonic map flow is  the bi-harmonic map flow on $4$-dimensional manifolds.
 Liu-Yin  in \cite{LY2} established the no-neck result of  a sequence of biharmonic maps. Later, Liu-Yin \cite{LY} generalized the no-neck result  to  a sequence of approximate biharmonic maps. By combining the no-neck result with a construction of a proper target manifold, they  introduced a concept of  width of bi-harmonic maps in the covering space   to show that the bi-harmonic map flow blows up in finite time. These results provide  a skeleton for the proof of Theorem \ref{main_theorem}.

This paper is organized as follows.  In   Section 2, we show   asymptotical behavior of the solution of the $n$-harmonic flow as $t\to \infty$. In Section 3, we generalize Ding-Tian's estimate and apply it to prove the no-neck result for the $n$-harmonic flow. In   Section 4, we construct an example to prove Theorem \ref{main_theorem} and settle the Hungerb\"uhler  conjecture.

\section{some estimates and asymptotic behavior of the $n$-harmonic map flow}

In order to study asymptotic behavior of the $n$-harmonic map flow, we begin with some basic estimates.
We recall some results from \cite{Hung} on the $n$-harmonic map flow.
\begin{lem}  \label{energy} Let $u(t)$ be a regular solution to the $n$-harmonic map flow (\ref{n-flow}) in $M\times
 [0,T]$ with initial value $u(0)=u_0$.  Then for each $s$ with  $0<s\leq T$, we have
\begin{eqnarray} \label{Energy identity}
  \int_M\frac 1 n|\nabla u(s)|^{n} \,dv +
  \int_0^t \int_M  \left |\frac{\partial u}{\partial t}\right |^2 dv\,dt \leq  \int_M \frac 1 n| \nabla u_0|^{n}
  \,dv.\nonumber
\end{eqnarray}
\end{lem}

\begin{lem}\label{2n estimate} Let $u$ be a regular solution to the $n$-harmonic map flow \eqref{n-flow}.   Let $\eta$ be a cut-off function in $B_{r}$ such that $\eta =1$ in $B_{\frac r 2}$, $|\nabla \eta|\leq \frac{C}{r}$ and $|\eta|\leq 1$ in $B_{r}$.
 Then we have
  \begin{eqnarray}
        && \int_{B_r} |\nabla u|^{2n}\eta^{n}\,dv\\
        &&\leq C \left (\int_{B_r}|\nabla u|^{n}\,dv\right )^{\frac 2 n} \int_{B_r}\left(|\nabla^2  u|^2\,|\nabla u|^{2n-4}\,\eta^{n} +  |\nabla u|^n |\nabla\eta|^n \right)\,dv.\nonumber
\end{eqnarray}
and
   \begin{eqnarray}
	  &&\int_{B_r}|\nabla^2  u|^{2}|\nabla u|^{2n-4}\eta^{n}\,dv
        \\&&\leq C\int_{B_r} |\nabla u|^{2n}  \eta^{n}+ |\nabla u|^{n}(\eta^n+|\nabla \eta|^n)   \,dv.\nonumber
 \end{eqnarray}
\end{lem}
\begin{proof}
By using the H\"older and Sobolev inequalities, we have
  \begin{eqnarray}\label{2n}
        &&\int_{B_r} |\nabla u|^{2n}\eta^{n}\,dv=\int_{B_r} |\nabla u| \left(|\nabla u|^{2n-1}\eta^{n}\right)\,dv\\
        &&\leq \left (\int_{B_r}|\nabla u|^{n}\,dv\right )^{\frac 1 n}\left (\int_{B_r} \left(|\nabla u|^{2n-1}\eta^{n}\right) ^\frac{n}{n-1}\,dv\right )^{\frac{n-1}n}\nonumber\\
        &&\leq C \left (\int_{B_r}|\nabla u|^{n}\,dv\right )^{\frac 1 n}  \int_{B_r}|\nabla ( |\nabla u|^{2n-1}\eta^{n})|\,dv \nonumber\\
        &&\leq C \left (\int_{B_r}|\nabla u|^{n}\,dv\right )^{\frac 1 n}  \int_{B_r}\left(|\nabla^2  u|\,|\nabla u|^{2n-2}\,\eta^{n} + |\nabla u|^{2n-1}|\nabla \eta|\,\eta^{n-1}  \right)\,dv.\nonumber
  \end{eqnarray}
By Young's inequality, we have
\begin{eqnarray}\label{2n-2}
  && \quad\left (\int_{B_r}|\nabla u|^{n}\,dv\right )^{\frac 1 n} \int_{B_r} |\nabla^2  u|\,|\nabla u|^{2n-2}\,\eta^{n}  \,dv \\
  &&\leq \left (\int_{B_r}|\nabla u|^{n}\,dv\right )^{\frac 2 n} \int_{B_r}  |\nabla^2 u|^2\, |\nabla u|^{2n-4} \,\eta^{n} dv + \frac 12 \int_{B_r} |\nabla u|^{2n}\,\eta^{n}dv.\nonumber
\end{eqnarray}
Similarly, we have
\begin{eqnarray}\label{2n-1}
  &&\int_{B_r}|\nabla u|^{2n-1}|\nabla \eta|\, \eta^{n-1}\,dv = \int_{B_r} |\nabla u |^n |\nabla u|^{n-1} |\nabla  \eta|\,\eta^{n-1}\,dv\\
  &&\leq C \int_{B_r} |\nabla u|^n |\nabla\eta|^n\eta^{n-1}\,dv+  C\int _{B_r} |\nabla u|^{2n}\eta ^{n-1}\,dv.\nonumber
\end{eqnarray}
Combining \eqref{2n}, \eqref{2n-2} with \eqref{2n-1}, we have
\begin{eqnarray}\label{2n_2}
&&\int_{B_r} |\nabla u|^{2n}\eta^{n}\,dv \nonumber\\ &&\leq C\left (\int_{B_r}|\nabla u|^{n}\,dv\right )^{\frac 2 n} \int_{B_r}\left(|\nabla^2  u|^2\,|\nabla u|^{2n-4}\,\eta^{n} +  |\nabla u|^n |\nabla\eta|^n \right)\,dv.
\end{eqnarray}

Using the Ricci identity, we have
\[\nabla_k \nabla_l \left (|\nabla u|^{n-2}
\nabla u\right )=\nabla_l\nabla_k \left ( |\nabla u|^{n-2} \nabla
u\right ) +R_M\#\left ( |\nabla u|^{n-2} \nabla u\right )\]
 with the Riemannian curvature $R_M$. Integrations by parts twice yield that
\begin{eqnarray*}
        &&\int_{B_r}\left <\nabla_k (|\nabla u|^{n-2}\nabla_k u), \nabla_l(|\nabla u|^{n-2}\nabla_l u) \right >\eta^n\,dv\\
         &&=- \int_{B_r}\left <\nabla_l\nabla_k (|\nabla u|^{n-2}\nabla_k u),  |\nabla u|^{n-2}\nabla_l u \right >\eta^n\,dv
         \\&& \quad - \int_{B_r}\left <  \nabla_k(|\nabla u|^{n-2}\nabla_k u), |\nabla u|^{n-2}\nabla_l u  \right >\nabla_l \eta^n\,dv\\
         &&=\int_{B_r}\left <\nabla_l (|\nabla u|^{n-2}\nabla_k u), \nabla_k(|\nabla u|^{n-2}\nabla_l u) \right >\eta^n\,dv\\
         &&\quad + \int_{B_r}\left <R_M\# (|\nabla u|^{n-2}\nabla_k u),  |\nabla u|^{n-2}\nabla_l u  \right >\eta^n\,dv\\
         && \quad + \int_{B_r}\left <  \nabla_l(|\nabla u|^{n-2}\nabla_k u), |\nabla u|^{n-2}\nabla_l u  \right >\nabla_k \eta^n\,dv\\
      && \quad - \int_{B_r}\left <  \nabla_k(|\nabla u|^{n-2}\nabla_k u), |\nabla u|^{n-2}\nabla_l u  \right >\nabla_l \eta^n\,dv.
  \end{eqnarray*}
Note that
 \begin{eqnarray}\label{Est-2}
        && \int_{B_r}\left <\nabla_l (|\nabla u|^{n-2}\nabla_k u), \nabla_k(|\nabla u|^{n-2}\nabla_l u)\right >\eta^n\,dv\\
         &&= \int_{B_r} \sum|\nabla u|^{2n-4}|\nabla_{lk} u|^2\, \eta^n+\left <\nabla_l (|\nabla u|^{n-2})\nabla_k u, \nabla_k(|\nabla u|^{n-2})\nabla_l u\right > \eta^n\,dv\nonumber\\
         &&\,\,\,\,+2\int_{B_r}\left <\nabla_l (|\nabla u|^{n-2})\nabla_k u, |\nabla u|^{n-2}\nabla_{kl} u \right >\eta^n\,dv\nonumber\\
         &&\geq \int_{B_r} (|\nabla u|^{2n-4}|\nabla^2 u|^2 + |\nabla_k  (|\nabla u|^{n-2})\nabla_k u|^2)\eta^n\,dv\nonumber\\
         &&\,\,\,\,+2(n-2)\int_{B_r}|\nabla u|^{n-4}|\nabla |\nabla u||^2\eta^n\,dv.\nonumber
  \end{eqnarray}
Combining \eqref{2n_2} with \eqref{Est-2}, this implies
   \begin{eqnarray}\label{2,2n-4}
        &&\int_{B_r}|\nabla^2  u|^{2}|\nabla u|^{2n-4}\eta^{n}\,dv\\
         &&\leq  \int_{B_r}  |\nabla \cdot (|\nabla u|^{n-2}\nabla u)|^2 \eta^n\,dv
         +C \int_{B_r}|\nabla u|^{2n-2}\eta^{n-2}(\eta^2+|\nabla\eta |^2)\,dv\nonumber
        \\&&\leq C\int_{B_r} |\nabla u|^{2n}  \eta^{n}+ |\nabla u|^{n}(\eta^n+|\nabla \eta|^n)   \,dv.\nonumber
  \end{eqnarray}
We finish the proof by combining \eqref{2n_2} with \eqref{2,2n-4}.

\end{proof}

\begin{lem}
There exists a sufficiently small constant $\varepsilon_{1}>0$ such that if
$u$  is a regular solution of
(\ref{n-flow}) on $B_{2R_0}\left(  x_{0}\right) \times\left[t_0-2R_0^n, t_0\right]$ satisfying
\[
\sup_{t_0-2R_0^n\leq t\leq t_0}\int_{B_{2R_0} (x_0)}|\nabla   u(x,t)|^n\,dv <\varepsilon_{1},
\]
 we have
\begin{eqnarray*} \label{2.3}
        &&
\int_{t_0-2R_0^n}^{t_0}\int_{B_{R_0}\left(  x_0\right)  }\left\vert
\nabla^{2}u\right\vert ^{2}\left\vert \nabla u\right\vert
^{2n-4}+ \left\vert \nabla
u\right\vert ^{2n} dvdt\\
&&\nonumber <C   \sup_{t_0- 2R_0^n\leq t\leq t_0}\int_{B_{2R_0} (x_0)}|\nabla   u(x,t)|^n\,dv
\end{eqnarray*}
for   some constant $C>0$.
\end{lem}
\begin{proof} Lemma 2.3 was  proved by Hungerb\"uhler by using  an extension of the Ladyzhenskaya-Solonnikov-Nikolaevna inequality (see Lemma 5 of \cite{H}). Herewith, we would like to give a slightly different approach by using Lemma  \ref{2n estimate}.

Multiplying (\ref {n-flow}) by $\phi^n\nabla \cdot (|\nabla u|^{n-2}\nabla u)$ and using Lemma  \ref{2n estimate} by choosing a sufficiently small $\varepsilon_1$ in above inequalities yields that
 \begin{eqnarray*}
        && \int_{B_{2R_0}(x_0)}  |\nabla \cdot (\left|\nabla u\right|^{n-2}\nabla u)|^2 \phi^n\,dv\\
        &&\leq \frac 12 \int_{B_{2R_0}(x_0)}  |\nabla \cdot (|\nabla u|^{n-2}\nabla u)|^2 \phi^n\,dv+C\int_{B_{2R_0}(x_0)} \left(\left|\frac {\partial u}{\partial t}\right|^2 +|\nabla u|^{2n}\right) \phi^n\,dv
        \\
        &&\leq \frac 3 4\int_{B_{2R_0}(x_0)}  |\nabla \cdot (|\nabla u|^{n-2}\nabla u)|^2 \phi^n\,dv + C\left(1+\frac 1 {R_0^n}\right)\int_{B_{2R_0}(x_0)}  |\nabla   u|^{n}\,dv\\
        &&+ C\int_{B_{2R_0}(x_0)} \left|\frac {\partial u}{\partial t}\right|^2 \,dv.
  \end{eqnarray*}

Together with Lemma \ref{2n estimate}, we obtain
\begin{eqnarray}\label{2n,2,2n-4}
  && \int_{B_{R_0}(x_0)} |\nabla u |^{2n}+|\nabla^2 u |^2 |\nabla u |^{2n-4}\,dv  \\
  &&\leq C \int_{B_{2R_0}(x_0)}  (R_0^{-n}+1) |\nabla   u |^{n}+ \left|\frac {\partial u}{\partial t}\right|^2 \,dv\nonumber\\
  &&\leq C(R_0^{-n}+1) E_n(u_0)+C \int_{B_{2R_0}(x_0)} \left|\frac {\partial u}{\partial t}\right|^2 \,dv.\nonumber
    \end{eqnarray}

 \end{proof}

\begin{lem} \label{lem:eregularity} Let $u$ be a regular solution to (\ref{n-flow}). Then  there exists a positive constant
    $\varepsilon_1$ such that  if for some $R_0>0$   the inequality
\[
\sup_{t_0-2R_0^n\leq t\leq t_0}\int_{B_{2R_0} (x_0)}|\nabla   u(x,t)|^n\,dv <\varepsilon_{1}
\]
 holds,
we have
\[\sup_{[t_0-  R_0^n,  t_0] \times B_{R_0}(x_0)} |\nabla u|^{n}\,dv \leq C  R_0^{-n},\]
where  $C$ depending on $M$ is a constant independent of $R_0$.
\end{lem}
\begin{proof} The proof   is  due to Hungerb\"uhler in  \cite{H} for $R_0=1$. If $R_0\neq 1$, one can prove it by a re-scaling argument.
 \end{proof}

\begin{lem}\label{2.5}
    Let $u:M\to N$ be a regular solution to the equation (\ref {n-flow}). Then
   there is a small constant $\varepsilon_1>0$ such that if  the inequality
\[
\sup_{t_0-R_0^n\leq t\leq t_0}\int_{B_{2R_0} (x_0)}|\nabla   u(x,t)|^n <\varepsilon_{1},
\] holds for some positive $R_0$,
then $\|u\|_{C^{1,\alpha}([t_0-\frac 1 2 R_0^n,  t_0] \times B_{R_0}(x_0))}$ is bounded by a constant depending on $E(u_0)$ and $R_0$.
\end{lem}
\begin{proof} As pointed out by Hungerb\"uhler in  \cite{H}, we can apply the result of DiBenedetto-Friedman \cite {DF1} to obtain a bound of    $\|u\|_{C^{1,\alpha}([t_0-\frac 1 2 R_0^n,  t_0] \times B_{R_0}(x_0))}$.
 \end{proof}

\begin{lem} (Local energy inequality under small condition) \label{small}
Let $u$  be a regular solution of
(\ref{n-flow}) on $B_{2R_0}\left(  x_{0}\right) \times\left[ 0,T\right]$.
There exists a sufficiently small constant $\varepsilon_{1}>0$ such that if
 \[
\sup_{t_0-T\leq t\leq t_0}\int_{B_{2R_0} (x_0)}|\nabla   u(x,t)|^n <\varepsilon_{1},
\]
then we have  for every $x\in B_{R}\left(  x_{0}\right)  $, any $R\leq R_0$ and any two constants $\tau$ and $s$ in $(t_0-T,t_0]$
 \begin{align*}  \int_{B_{R}(x_0) }|\nabla u|^n (\cdot ,s)\,dv\leq &\,\int_{B_{2R}(x_0)} |\nabla u|^n(\cdot , \tau )\,dv
+C\int_{s}^{\tau }\int_{B_{2R}(x_0)} |\partial_t u|^2\,dv\,dt\\ +&C\left (\frac {(\tau -s)}{R^n}\,\int_{B_{2R}(x_0)} |\nabla u|^n\,dv\, \int_{s}^{\tau
}\int_{B_{2R}(x_0)} |\partial_t u|^2\,dv\,dt \right )^{1/2} \end{align*}
for some constant $C$.
\end{lem}
\begin{proof}
Let $\phi$ be a cut-off function with support in $B_{2R_0}(x_0)$ such that $\phi =1$ in $B_{R_0}(x_0)$, $|\nabla \phi|\leq CR_0^{-1}$ and $|\phi|\leq 1$ in $B_{2R_0}(x_0)$.
Multiplying (\ref{n-flow}) by $\phi^n \partial_t u$, we have
\begin{eqnarray*}
        && \int_{B_{2R_0}(x_0)} |\frac {\partial u}{\partial t} |^2\phi^n\,dv=\int_{B_{2R_0}(x_0)} \left <\nabla \cdot (|\nabla u|^{n-2}\nabla u), \,\frac {\partial u}{\partial t}\right >\phi^n\,dv\\
        &&=-\int_{B_{2R_0}(x_0)} \left < |\nabla u|^{n-2}\nabla u,\, \frac {\partial \nabla u}{\partial t} \phi^n +\frac {\partial u}{\partial t}\phi^{n-1}\nabla \phi \right >\,dv\\
        &&\geq -\frac 1 n \frac d{dt} \int_{B_{2R_0}(x_0)}|\nabla u|^n \phi^n\,dv- C \int_{B_{2R_0}(x_0)}|\nabla u|^{n-1}  |\frac {\partial u}{\partial t} |\phi^{n-1}|\nabla \phi|  \,dv.
  \end{eqnarray*}
Note
\begin{eqnarray*}
        &&
\int_{B_{2R_0}(x_0)}|\nabla u|^{n-1}  |\frac {\partial u}{\partial t} |\phi^{n-1}|\nabla \phi|  \,dv\\
\leq &&\left (\int_{B_{2R_0}(x_0)} |\frac {\partial u}{\partial t} |^2\phi^n\,dv\right )^{1/2} \left (\int_{B_{2R_0}(x_0)}|\nabla u|^{2n-2}  \phi^{n-2}\,|\nabla \phi|^2  \,dv\right )^{1/2}
  \end{eqnarray*}
since
\[\int_{B_{2R_0}(x_0)}|\nabla u|^{2n-2}  \phi^{n-2}\,|\nabla \phi|^2  \,dv\leq C\int_{B_{2R_0}(x_0)} |\nabla u|^{2n}  \phi^{n}+ |\nabla u|^{n}|\nabla \phi|^n   \,dv.\]
Therefore, the claim is proved.
 \end{proof}

\begin{Proposition}\label{Theorem 3} Let $u$ be a regular solution to (\ref{n-flow}) in $M\times [0,\infty)$. For a sequence $t_i\to\infty$, there is a sub-sequence, still denoted by $t_i\to\infty$, such that
$u(\cdot , t_i)$ converges to an $n$-harmonic maps $u_{\infty}$ locally in $C^{1,\a }(M\backslash \{x^1,...,x^J\}; N)$ with some positive $\a<1$, where $u_{\infty}$ can be extended regularly on $M$.
\end{Proposition}
\begin{proof}  By Lemma \ref{energy},  we know that $\int_0^{\infty }\int_M
|\partial_tu|^2 \,dvdt$ is finite, so we may choose a sub-sequence
$\{t_i\}$ such that as $t_i\to\infty $, $\partial_t u (\cdot , t_i)\to 0$
strongly in $L^2(M)$ and $\int_{t_i-1}^{t_i}\int_M |\partial_t u (\cdot , t)|^2\,dv\,dt\to 0$.  Moreover,
 there is a constant $\varepsilon_0>0$ such that  the singular points $\{x^1, ..., x^J\}$ are defined by the condition
\[ \limsup_{t_i \to \infty}  E_n(u(t_i); B_R(x^k)) \geq \varepsilon_0 \] for any $R\in (0, 2R_0]$  with some fixed $R_0>0$.

For each $x_0\in M\backslash \{x^1, ..., x^J\}$, there is a sufficiently small $R_0>0$ such that $B_{2R_0} (x_0)\subset  M\backslash \{x^1, ..., x^J\}$ and for all $i$,
\[\int_{B_{2R_0} (x_0)}|\nabla   u(x,t_i)|^n \,dv<\varepsilon_{0}\leq \frac {\varepsilon_1}2, \]
where $\varepsilon_1$ is the constant defined in Lemma \ref{small}.

By Lemma \ref{small}, we have  for any $s\in [t_i-2R_0^n, t_i]$ and  for sufficiently large $i$
\begin{align*}  \int_{B_{R_0}(x_0) }|\nabla u|^n (\cdot ,s)\,dv\leq &\,\int_{B_{2R_0}(x_0)} |\nabla u|^n(\cdot , t_i )\,dv
+C\int^{t_i}_{t_i-1}\int_{B_{2R_0}(x_0)} |\partial_t u|^2\,dv\,dt\\ +&C\left (\frac {(t_i -s)}{R_0^n}\,\int_{B_{2R_0}(x_0)} |\nabla u|^n\,dv\,  \int^{t_i}_{t_i-1}\int_{B_{2R_0}(x_0)} |\partial_t u|^2\,dv\,dt \right )^{1/2} \\
&<  \varepsilon_{1}. \end{align*}

By Lemma \ref{lem:eregularity}, we have
 \[\sup_{t\in [t_i-   R_0^n,  t_i],\,  x\in  B_{R_0}(x_0)} |\nabla u|^{n}(x , t  )\,dv \leq C  R_0^{-n}.\]
 Then using Lemma \ref{2.5}, there is a uniform bound of  $\|u(\cdot ,t_i)\|_{C^{1, \a }(B_{\frac 1 2 R_0}(x_0))}$, so $u(x,t_i)$ convergence to $u_{\infty}$ in $C^{1, \b }(B_{\frac 1 2 R_0})$  and hence in $C_{loc}^{1,\b}(M\backslash \{x^1, ..., x^J\})$ with  $\b<\a$, where $u_{\infty}\in C_{loc}^{1,\b}(M\backslash \{x^1, ..., x^J\})$ is an $n$-harmonic map.
By the removable singularities of  an $n$-harmonic map, $u_{\infty}$ can be extended to $C^{1,\a}(M)$.
\end{proof}

\section{No neck  result between the limiting map and bubbles as $t\to\infty$}

In this section, we generalize the no-neck result of Qing-Tian \cite{QT} to the case of the $n$-harmonic flow.
As suggested by Struwe \cite{St} and Qing \cite{Q}, the existence of   solutions of the heat flow for harmonic maps can be proved by a method of ``Palais-Smale sequences'' with tension fields $\tau(u)\in L^2$.
In the context of $n$-harmonic maps, the tension field $\tau(u)$ of $u$ is defined as follows:
\begin{equation}\label{approx_n_harmonic}
 \tau(u):=  \frac{1}{\sqrt{|g|}}\frac{\partial}{\partial x_{i}}\left[  |\nabla
u|^{n-2}g^{ij}\sqrt{|g|}\frac{\partial}{\partial x_{j}}u\right]+
|\nabla
u|^{n-2}A(u)(\nabla u,\nabla u),
\end{equation}
where $A$ is the second fundamental form of $N$.

If $\tau(u)=0$, $u$ is  an $n$-harmonic map. When $\tau(u) \in L^2(M)$ and $u$ satisfies an extra smoothness assumption in \eqref{approx_n_harmonic}, we define $u$ to be a regular approximated $n$-harmonic map as follow (See \cite{WW} for details):

\begin{definition}
We define a map $u\in W^{1,n}(M; N) \cap C^0(M; N)$ to be a regular approximated $n$-harmonic map if it satisfies the following conditions:
	\begin{enumerate}
 	 \item $\nabla\left(|\nabla u|^{\frac{n-2}{2}} \nabla u\right)\in L^2(M)$;
 	 \item There exist $\varepsilon>0$, $\alpha\in (0,1)$ and $C>0$ depending only on $M$, $N$ and $\|\tau(u)\|_{L^2}$ such that for any $B_{2_r }\in M$ and  $E_n(u; B_{2_r })\leq \varepsilon$, then
 \[ u\in C^\alpha ( B_r; N)\qquad \text{and}\qquad [u]_{C^\alpha(B_r(x))}\leq C.\]
	\end{enumerate}
\end{definition}

Let $\{u_i\}$ be a sequence of  regular approximated $n$-harmonic maps with   uniform  bounds of $E_n(u_i)$ and $\|\tau(u_i)\|_{L^2(M)}$.  Wang-Wei \cite {WW} proved that $\{u_i\}$  converges to an $n$-harmonic map $u_{\infty}$ strongly in $W_{loc}^{1,q} (M\backslash \{x^1, \cdots, x^L\})$ for any $q<2n$, where $u_{\infty}$ can be extended  to $C^{1,\a}(M)$. By reducing multi bubbles into a single bubble, they proved that there are a finite number of $n$-harmonic maps $\omega_{k,l} $   on $S^{n}$ with $k=1,\cdots, L$ and $l=1,\cdots J_k$  such that
\[\lim_{t_i \nearrow \infty }  E_n(u_i; M)= E_n(u_{\infty}; M)+\sum_{k=1}^{L} \sum_{l=1}^{J_k}   E_n  (\omega_{k,l}  , S^{n}). \]
Then we have
\begin{theorem}\label{no_neck2}  Let $\{u_i\}$ be the sequence of  regular approximated $n$-harmonic maps with   uniform  bounds of $E_n(u_i)$ and $\|\tau(u_i)\|_{L^2(M)}$,  and  let $\omega_{k,l} $ be the above bubbles. Then
there is no neck between the limiting map $u_{\infty}$ and bubbles $\omega_{k,l}$; \\i.e. the image  $$u_{\infty}(M)\cup \bigcup_{k,l}\omega_{k,l}(S^n)$$ is a connected set.
\end{theorem}

We begin with the following $\varepsilon$-regularity estimate for approximated $n$-harmonic maps. In particular, we generalize the Ding-Tian estimate (see \cite{DT}, Lemma 2.1), which is a crucial estimate to the proof of no-neck result.

\begin{lem}\label{Osc} For  $n\geq 2 $,
   let $u\in W^{1,n}(M,N) \cap C^0(M,N)$  be an approximated $n$-harmonic map. Then there exists a small constant $\varepsilon>0$  such that if $E_n(u, B_r)\leq \varepsilon$ then
   \begin{align}\label{osclemma}
    \|u\|_{osc(B_\frac{r}{2})} \leq C \left(\int_{B_{r}}  |\nabla   u |^{n}\,dv\right)^\frac{1}{2(n-1)} + Cr^\frac{n}{2(n-1)}\left( \int_{B_{r}}|\tau(u)|^2 \,dv \right)^\frac{1}{2(n-1)}.
   \end{align}
\end{lem}

\begin{proof}
Let  $\phi$ be a cut-off function in $C_0^{\infty}(B_r )$ with $\phi \equiv 1$ in $B_{\frac r2}$ and $|\nabla \phi|\leq Cr^{-1}$ and set $\bar u=\frac 1{|B_{\frac{3}{4}r}|}\int_{B_{\frac{3}{4}r}} u \,dv$. For a sufficient small $a>0$, we apply Theorem 7.17 in \cite{GT} with $p= \frac{2n(n-1)}{n-2+a}>n$, $\gamma =1-\frac np$, and the Poincar\'{e} inequality to obtain

\begin{align}
\|u\|_{osc(B_\frac{r}{2})} = &\sup_{x,y\in B_{\frac r2}}  | u(x)-u(y) |\leq  2\sup_{x\in B_{\frac{3}{4}r}} |\, (u(x)-\overline{u})\,\phi(x) \,| \nonumber \\
&\leq Cr^{1-\frac {n-2+a}{2(n-1)}} \left (\int_{B_{\frac{3}{4}r} } |\nabla [(u-\overline{u})\,\phi]|^{\frac{2n(n-1)}{n-2+a}}\,dv\right )^\frac{n-2+a}{2n(n-1)}\nonumber\\
&\leq C r^{\frac{n-a}{2(n-1)}}\left (\int_{B_{\frac{3}{4}r} } |\nabla u(x)|^{\frac{2n(n-1)}{n-2+a}}\,dv\right )^\frac{n-2+a}{2n(n-1)} \nonumber\\
& +Cr^{\frac{n-a}{2(n-1)}} \left (\int_{B_{\frac{3}{4}r} }  |(u(x)-\overline{u}) \nabla \phi| ^{\frac{2n(n-1)}{n-2+a}}\,dv\right )^\frac{n-2+a}{2n(n-1)}\nonumber\\
&\leq Cr^{\frac{n-a}{2(n-1)}} \left (\int_{B_{\frac{3}{4}r}} \left\vert \nabla
u\right\vert ^{\frac{2n(n-1)}{n-2+a}} \,dv\right )^\frac{n-2+a}{2n(n-1)}\nonumber\\
 & =C r^{\frac{n-a}{2(n-1)}} \left(\int_{B_{\frac{3}{4}r}}  \left|\, |\,\nabla u |^{n-2}\nabla u \right|^\frac{2n}{n-2+a} \,dv \right )^{\frac{n-2+a}{2n} \frac{1}{(n-1)}}.
 \end{align}

 By using the Sobolev-Poincar\'{e} inequality on $B_1$ (Page 174 in \cite{GT}), we have for $p<n$
 \[\left (\int_{B_1} |f-f_{B_1} |^{p^*}\,dv\right )^{1/p^*}\leq C \left (\int_{B_1} |\nabla f  | ^{p}\,dv\right )^{p}. \]
Choosing $p= \frac {2n}{n+a} <2 $ such that $q =\frac{2n}{n-2+a}=\frac {np}{n-p}=p^*$ and re-scaling  from $B_1$ to $B_r$ and using H\"older's inequality,  we have
\begin{align}
&\quad\left( \frac{r^{(n-1)q}}{r^n}\int_{B_r} |\,|\nabla u|^{n-2} \,\nabla u | ^{q}\,dv \right)^{\frac{1}{q}}\\
&\leq C \left(\frac{r^{2n}}{r^n}\int_{B_r}|\,\nabla(|\nabla u|^{n-2}\,\nabla u) |^2 \,dv\right)^{1/2}
 + \frac{C}{r}\int_{B_r}|\nabla u|^{n-1}\,dv.\nonumber
\end{align}
By the H\"older inequality, we have
\[\left (\frac{1}{r}\int_{B_r}|\nabla u|^{n-1}\,dv\right )^{\frac 1{n-1}} \leq  C\left(\int_{B_{\frac{3}{4}r}} |\nabla u|^{n}\,dv\right)^{\frac{1}{n}}.\]
Then
substituting (3.4) into (3.3), we obtain
 \begin{align}
&\qquad\|u\|_{osc(B_\frac{r}{2})}  \leq C   r^{\frac{n-a}{2(n-1)}} \left(\int_{B_{\frac{3}{4}r}}  \left|\, |\,\nabla u |^{n-2}\nabla u \right|^q \,dv \right )^{\frac 1 q \frac{1}{(n-1)}} \\
& \leq C    r^{\frac{n-a}{2(n-1)}}r^{\frac n{q(n-1)}-1} r^\frac{n}{2(n-1)}\left( \int_{B_r}|\,\nabla(|\nabla u|^{n-2}\,\nabla u) |^2 \,dv\right)^\frac{1}{2(n-1)}\nonumber\\
 &+ C r^{\frac{n-a}{2(n-1)}} r^{\frac n{q(n-1)}-1} \left (\frac{1}{r}\int_{B_r}|\nabla u|^{n-1}\,dv\right )^{\frac 1{n-1}} \nonumber\\
 &= C r^{\frac{n}{2(n-1)}} \left(\int_{B_{\frac{3}{4}r}} |\nabla(|\nabla u|^{n-2} \nabla u)|^2 \,dv\right)^\frac{1}{2(n-1)}+ C \left (\frac{1}{r}\int_{B_r}|\nabla u|^{n-1}\,dv\right )^{\frac 1{n-1}} \nonumber \\
 &\leq C  \left(r^n\int_{B_{\frac{3}{4}r}} |\nabla(|\nabla u|^{n-2} \nabla u)|^2 \,dv\right)^\frac{1}{2(n-1)} + C\left(\int_{B_{\frac{3}{4}r}} |\nabla u|^{n}\,dv\right)^{\frac{1}{n}} \nonumber\end{align}
 by  noting that $\frac{n-a}{2(n-1)}+\frac n{q(n-1)}-1=0$ with $q =\frac{2n}{n-2+a}$.

\vspace{10pt}
Multiplying (\ref {approx_n_harmonic}) by $\nabla \cdot (|\nabla u|^{n-2}\nabla u)\,\eta^n$, we have
\begin{align}\label{div-estimate}
  \int_{B_r}|\nabla \cdot (|\nabla u|^{n-2}\nabla u)|^2 \eta^n\,dv\leq  \int_{B_r}|\nabla \cdot (|\nabla u|^{n-2}\nabla u)|(\tau(u) +C |\nabla u |^n ) \eta^n\, dv. \nonumber
\end{align}
Now, using Young's inequality, we have
 \begin{eqnarray*}
         \int_{B_r}  |\nabla \cdot (|\nabla u|^{n-2}\nabla u)|^2 \eta^n \,dv\leq
          C\int_{B_r} (|\tau(u)|^2 +|\nabla u|^{2n}) \eta^n \,dv.
  \end{eqnarray*}
Using Lemma \ref{2n estimate} again, it yields that
\begin{eqnarray*}
 \int_{B_{\frac{3}{4}r}} |\nabla u |^{2n}+|\nabla^2 u |^2 |\nabla u |^{2n-4}\,dv \leq C \int_{B_{r}}  (1+r^{-n}) |\nabla   u |^{n}+ |\tau(u)|^2 \,dv.
    \end{eqnarray*}
Therefore
\begin{eqnarray*}
  && \left (r^n \int_{B_{\frac 34 r}} |\nabla u |^{2n}+|\nabla^2 u |^2 |\nabla u |^{2n-4} \,dv\right )^\frac{1}{2(n-1)} \\
  &&\leq C \left (\int_{B_{r}}  |\nabla   u |^{n}+ r^n |\tau(u)|^2 \,dv\right )^\frac{1}{2(n-1)}\\
  &&\leq C \left(\int_{B_{r}}  |\nabla   u |^{n}\,dv\right)^\frac{1}{2(n-1)} + C r^\frac{n}{2(n-1)} \left(\int_{B_{r}}|\tau(u)|^2\,dv \right)^\frac{1}{2(n-1)}.
\end{eqnarray*}
We finish the proof by putting these estimates together.

\end{proof}

To analyze the behavior of approximated $n$-harmonic maps on the neck region, we need the following Pohozaev type inequality, which was proved in \cite{WW}:

\begin{lem}\label{Pohozaev} For $n\geq 2$, let $u\in W^{1,n}(M, N) \cap C^{1,\alpha}(M, N)$ to be a regular approximated $n$-harmonic map with tension field $\tau(u) \in L^2(M)$. Then, for any ball $B_r\subset M $, we have
\begin{equation}
   \int_{\partial B_r} | \nabla u | ^n \,ds\leq C(n) \left( \int_{\partial B_r} |\nabla_T \,u|^n\,ds+ \int_{B_r} |\tau(u)| \, |\nabla\, u|\,dv \right),
\end{equation}
\noindent where $\nabla_T u$ is the tangential gradient on the boundary $\partial B_r$.
\end{lem}
\begin{proof} For completeness,
we sketch the proof here. Multiplying (\ref{approx_n_harmonic}) by $x\,\cdot \nabla u$ and integrating over $B_r$, we have
\begin{align*}
	&\quad \int_{B_r}\left\langle\tau(u),\, x\cdot \nabla u\right\rangle \,dv \\
&= \frac{1}{n}\int_{B_r}\left\langle x\, , \, \nabla(|\nabla\, u|)^n\right\rangle \,dv+ \int_{B_r} | \nabla u| ^n\,dv - r\int_{\partial B_r} |\nabla u|^{n-2}\, \left| \frac{\partial u}{\partial r}\right|^2\, ds\\
	&=\frac{r}{n} \int_{\partial B_r} |\nabla u|^n \, ds -  r\int_{\partial B_r} |\nabla u|^{n-2}\, \left| \frac{\partial u}{\partial r}\right|^2\, ds,
\end{align*}
where we use the fact that
\begin{align*}
 \int_{B_r} \left\langle x\, , \, \nabla(|\nabla\, u|)^n\right\rangle\,dv = r\int_{\partial B_r} |\nabla u|^n \, ds - n\int_{B_r} |\nabla u|^n\,dv
\end{align*}
and $|\nabla u|^2 = \left| \frac{\partial u}{\partial r}\right|^2 +  \left| \nabla _Tu\right|^2.$

Rearranging the inequality and by adding $(n-1)\int_{\partial B_r} |\nabla u|^{n-2} \left| \nabla _Tu\right|^2 ds $ to the both sides we have
\begin{align*}
(n-1) \int_{\partial B_r} |\nabla u|^{n} \,ds &\leq n\int_{B_r} |\tau(u) | |\nabla u| \,dv+ n \int_{\partial B_r} |\nabla u|^{n-2}\left(  \left| \nabla _T u\right|^2\right) ds.
\end{align*}
Then the claim follows from using Young's inequality.
\end{proof}

Now we prove Theorem 3.
\begin{proof}
By using the standard bubbling arguments as in \cite{DT} and \cite{WW}, one can reduce multiple bubbles to a
single bubble.
We assume that $0$ is the single blowing up point of $\{u_i\}$ and there is only one bubble in $B_1$.
Then, we follow the approach of \cite{QT} and \cite{LZ} to extend the no-neck result to the case of the $n$-harmonic map flow.

Suppose $r_n \, R = 2^{-j_n}$ and  $\delta = 2^{-j_0}$  for any $j_0<j<j_n$. Then, we denote
$$L_j = \min \{j-j_0 ,\, j_n-j\}\quad\text{and}\quad P_{j,t}= B_{2^{t-j}}\setminus B_{2^{-t-j}} \quad \text{for} \quad t\in(0, L_j].$$

For sufficiently large $i$, we assume that
\begin{align}\label{neck-energy-control}
  E_n(u_i, B_{2^{1-j}} \setminus B_{2^{-j}})\leq \varepsilon^{2(n-1)}, \qquad \text{for any}\, j_0\leq j \leq j_n.
\end{align}

Let
$$ h_{i,j,t}(2^{\pm t-j}) = \frac{1}{|S^{n-1}
|}\int_{S^{n-1}} u_i(2^{\pm t-j},\theta)\, d\theta$$
and
\begin{equation}
  h_{i,j,t}(r) = h_{i,j,t}(2^{t-j})+(h_{i,j,t}(2^{-t-j}) - h_{i,j,t}(2^{t-j}))\frac{\ln (2^{-t+j}\, r)}{-2t\, \ln 2\hfill}.
\end{equation}
Note that the tangential derivative of $h_{i,j,t}(r)$ is zero in $n$-dimensional spherical coordinates. Therefore, the Laplace operator can be reduced to the following form:
\begin{align*}
  \Delta h_{i,j,t} = \frac{d^2\,h_{i,j,t} }{d^2r} + \frac{n-1}{r}\, \frac{d h_{i,j,t} }{dr},
\end{align*}
which yields that
\begin{align*}
  \rm div(|\nabla h_{i,j,t}|^{n-2}\,\nabla\, h_{i,j,t}) &= \left| \frac{d \,h_{i,j,t}}{dr}\right|^{n-2}\left(\frac{d^2\,h_{i,j,t} }{d^2r} + \frac{n-1}{r}\, \frac{d h_{i,j,t} }{dr}\right)\\
  &\,+ \frac{n-2}{2}\left| \frac{d \,h_{i,j,t}}{dr}\right|^{n-4}\frac{d h_{i,j,t}}{dr} \, \frac{d}{dr}\left| \frac{d h_{i,j,t}}{dr}\right|^2\\
  &=(n-1)\left| \frac{d \,h_{i,j,t}}{dr}\right|^{n-2}\left( \frac{d^2\,h_{i,j,t} }{d^2r}+\frac{1}{r}\frac{d\,h_{i,j,t} }{dr}\right)\\
  &=0.
\end{align*}
This implies that $h_{i,j,t}(r)$ is also a symmetric $n$-harmonic map for $r\in [2^{-t-j},\, 2^{t-j}]$.
By the well-know result of the $n$-Laplace operator, we note that
\begin{align*}
  &\int_{P_{j,t}} |\nabla(u_i-h_{i,j,t})|^{n} \,dv\\
  &
  \leq C\int_{P_{j,t}}\left\langle(|\nabla u_i|^{n-2}\, \nabla u_i - |\nabla h_{i,j,t}|^{n-2} \, \nabla h_{i,j,t}), \, \nabla (u_i-h_{i,j,t}) \right\rangle \,dv.
\end{align*}
for some constant $C>0$.

By integration by parts, we have
\begin{align*}
  &\int_{P_{j,t}}\left\langle(|\nabla u_i|^{n-2}\, \nabla u_i - |\nabla h_{i,j,t}|^{n-2} \, \nabla h_{i,j,t}), \, \nabla (u_i-h_{i,j,t}) \right\rangle \,dv.\\
  &= - \int_{P_{j,t}} \left\langle div( |\nabla u_i|^{n-2}\, \nabla u_i ), \, (u_i - h_{i,j,t})\right\rangle \,dv\\
  &+ \int_{\partial P_{j,t}} \left\langle(|\nabla u_i |^{n-2} (u_i)_r - |\nabla h_{i,j,t}|^{n-2} (h_{i,j,t})_r),\, (u_i-h_{i,j,t})\right\rangle \,dv\\
  &=  \int_{P_{j,t}} \left\langle \left(A(u_i)(\nabla\,u_i, \nabla u_i)  |\nabla u| ^{n-2} + \tau(u_i)\right),\,  (u_i - h_{i,j,t})\right\rangle \,dv\\
  &+ \int_{\partial P_{j,t}} \left\langle \left(|\nabla u_i |^{n-2} (u_i)_r - |\nabla h_{i,j,t}|^{n-2} (h_{i,j,t})_r\right),\, (u_i-h_{i,j,t})\right\rangle \,dv.
\end{align*}
By Lemma \ref{Osc}, we obtain
\begin{align}\label{CPjt}
\|u_i-h_{i,j,t}\|_{C^0(P_{j,t})}&\leq \|u_i-h_{i,j,t}(2^{j-t})\|_{C^0(P_{j,t})} + \|u_i-h_{i,j,t}(2^{-j-t})\|_{C^0(P_{j,t})} \nonumber\\
&\leq 2 \| u_i\|_{osc(P_{j,t})}\nonumber\\
&\leq  C \left(\int_{P_{j-1,t} \cup P_{j,t} \cup P_{j+1,t}}  |\nabla   u |^{n}\right)^\frac{1}{2(n-1)}\nonumber\\
&\quad + C \,(2^{t-j+1})^\frac{n}{2(n-1)}\left( \int_{B_{2^{t-j+1}}}|\tau(u)|^2  \right)^\frac{1}{2(n-1)}\nonumber\\
&\leq C\left(\varepsilon + \delta^{\frac{n(t+1)}{2(n-1)}} \right)\leq C\varepsilon.
\end{align}

By \eqref{CPjt}, we have
\begin{align*}
  &\int_{P_{j,t}} |\nabla(u_i-h_{i,j,t})|^n \, dv \\
  & \leq C\int_{P_{j,t}}\left < |\nabla u_i|^{n-2}\, \nabla u_i - |\nabla h_{i,j,t}|^{n-2} \, \nabla h_{i,j,t}, \,\nabla (u_i-h_{i,j,t})\right >\, dv\\
  &\leq C  \int_{P_{j,t}}\left < (A(u_i)(\nabla\,u_i, \nabla u_i) \,|\nabla u| ^{n-2} +\tau(u_i)), \,(u_i-h_{i,j,t}) \right >\, dv \\
  & + C\int_{\partial P_{j,t}} \left <|\nabla u_i |^{n-2}\, (u_i)_r,\,  (u_i-h_{i,j,t})\right >\,ds\\
  &\leq C \left( \varepsilon\int_{ P_{j,t}}\, |\nabla u_i| ^{n} \, dv + \varepsilon\|\tau(u_i)\|_{L^2(P_{j,t})} \, 2^{(t-j)  \frac{n}{2}}\right)\\
  & + C \int_{\partial P_{j,t}} |\nabla u_i |^{n-2}\, |(u_i)_r| \,\, |u_i-h_{i,j,t}|\,ds \\
  &=C (\varepsilon (I_1 +2^{(t-j)\,   \frac{n}{2}}) +  I_2),
\end{align*}
where we set
\[I_1: =f_j(t) =\int_{P_{j,t}} |\nabla u_i|^n \,dv \qquad\text{and}\qquad I_2 :=  \int_{\partial P_{j,t}} |\nabla u_i |^{n-2} \,\,|(u_i)_r|  \,\, |u_i-h_{i,j,t}| \, ds.\]

Using the fact that $\frac{d\,2^x}{dx} = \ln2 (2^x)$, this implies
$$f_j'(t)= \ln 2 \, \left( 2^{t-j} \int_{ \{ 2^{t-j}\} \times S^{n-1}} |\nabla u_i|^n \, ds +2^{-t-j} \int_{ \{ 2^{-t-j}\} \times S^{n-1}} |\nabla u_i|^n  \, ds\right). $$
By  the Poincar\'{e} inequality and   H\"{o}lder's inequality, we have
\begin{align} \label{I2}
  I_2&=  \int_{\partial P_{j,t}} (|\nabla u_i |^{n-2} \,\,|(u_i)_r| )\,\, |u_i-h_{i,j,t}| \, ds\\
  &\leq   \int_{ \{ 2^{t-j}\} \times S^{n-1}} |\nabla u_i|^{n-1} |u_i-h_{i,j,t}| \, ds + \int_{ \{ 2^{-t-j}\} \times S^{n-1}} |\nabla u_i|^{n-1} |u_i-h_{i,j,t}| \, ds\nonumber\\
  & \leq  \left( \int_{ \{ 2^{t-j}\} \times S^{n-1}} (|\nabla u_i|^{n-1})^\frac{n}{n-1}  \, ds\right)^ \frac{n-1}{n} \left( \int_{ \{ 2^{t-j}\} \times S^{n-1}} |u_i-h_{i,j,t}|^n \, ds  \right)^\frac{1}{n}\nonumber\\
  &+  \left( \int_{ \{ 2^{-t-j}\} \times S^{n-1}} (|\nabla u_i|^{n-1})^\frac{n}{n-1}  \, ds\right)^ \frac{n-1}{n} \left( \int_{ \{ 2^{-t-j}\} \times S^{n-1}} |u_i-h_{i,j,t}|^n \, ds  \right)^\frac{1}{n}\nonumber\\
  & \leq  \left( \int_{ \{ 2^{t-j}\} \times S^{n-1}} (|\nabla u_i|)^{n}  \, ds\right)^ \frac{n-1}{n} \left( \int_{ \{ 2^{t-j}\} \times S^{n-1}} |u_i-h_{i,j,t}|^n \, ds  \right)^\frac{1}{n}\nonumber\\
  &+    \left( \int_{ \{ 2^{-t-j}\} \times S^{n-1}} (|\nabla u_i|)^{n} \, ds\right)^ \frac{n-1}{n} \left( \int_{ \{ 2^{-t-j}\} \times S^{n-1}} |u_i-h_{i,j,t}|^n \, ds  \right)^\frac{1}{n} \nonumber\\
  &\leq C f_j'(t).\nonumber
\end{align}

Note that $h_{i,j,t}$ is the average of $u_i$ over $S^{n-1} $. Then we can estimate the tangential energy by
\begin{align}\label{tangnet energy}
    \int_{P_{j,\,t}} |\nabla_T \,u_i|^n \, dv&\leq  \int_{P_{j,t}} \left(| (u_i - h_{i,j,t})_r |^2 +  |\nabla _T \,u_i|^2   \right)^\frac{n}{2} \, dv \\
    &= \int_{P_{j,t}} |\nabla(u_i-h_{i,j,t})|^n \, dv\nonumber\\
    &\leq C (\varepsilon I_1 +2^{(t-j)\,   \frac{n}{2}} +  I_2).\nonumber
\end{align}
By Lemma \ref{Pohozaev}, given a regular approximated $n$-harmonic map $u_i$, for all $r\in \left[\lambda_nR ,\delta\right]$,  we have
\begin{equation}\label{m-harmonic_maps_bounard_energy}
  \int_{\partial B_r} | \nabla u_i | ^n\,ds \leq  C(n) \left( \int_{\partial B_r}  |\nabla_T \,u_i|^n\,ds+ \int_{B_r} |\tau(u_i)| \, |\nabla u_i| \,dv\right).
\end{equation}
Integrating (\ref{m-harmonic_maps_bounard_energy}) in $r$ from $r=2^{-t-j}$ to $r=2^{t-j}$, using H\"{o}lder's inequality and (\ref{tangnet energy}), we obtain
\begin{align}\label{f}
      &\quad f_j(t)= \int_{P_{j,t}} |\nabla u_i|^n \,dv \\
      &\leq C\left( \int_{P_{j,\,t}}  |\nabla_T \,u_i|^n\,dv+ \int^{2^{t-j}}_{2^{-t-j}} \|\tau(u_i)\|_{L^{2}(B_r)}\, \|\nabla u_i\|_{L^2(B_r)}\,dr \right)\nonumber\\
      &\leq C  ( \varepsilon( I_1 +2^{(t-j)\,   \frac{n}{2}}) +  I_2)) + C \int^{2^{t-j}}_{0} \|\tau(u_i)\|_{L^{2}(B_r)}\, \|\nabla u_i\|_{L^n(B_r)}r^{\frac{n-2}2}\,dr \nonumber \\
      &\leq   C (\varepsilon I_1 +2^{(t-j)\,   \frac{n}{2}} +  I_2 ) \leq C (\varepsilon  f_j(t) +2^{(t-j)\,   \frac{n}{2(n-1)}}) +   C f_j'(t)  .\nonumber
\end{align}
Let $\lambda_n = \frac{n}{2(n-1)}\ln 2$. Choosing  $\varepsilon$ sufficiently small in (\ref{f}), we have
$$0\leq f'_t(t) -\frac{1}{C} f_j(t) +C e^{\lambda_n(t-j)}. $$
Now, assuming that $\lambda_n>\frac{1}{C}$ for a sufficiently large $C$, it implies
\begin{equation}\label{c_1}
  0\leq \left(e^{-\frac{t}{C}} f_j(t)\right)' +C e^{\lambda_n(t-j)}\, e^{-\frac{t}{C}}.
\end{equation}
Integrating (\ref{c_1}) in $t$ over $[2,L_j]$, this gives
\begin{align}\label{f_j(2)_estimate}
     f_j(2)&\leq C \left( e^{\frac{-L_j}{C}}\,f_j(L_j)  + e^{-\lambda_n\,j}\,  e^{\left(\lambda_n-\frac{1}{C}\right)\,L_j} \right)\nonumber\\
     &\leq C \left( e^{\frac{-L_j}{C}}\,f_j(L_j)  + e^{-\lambda_n\,j}\,  e^{\frac{-j}{C}} \right),
\end{align}
where we note that
$$ P_j= B_{2^{1-j}}\setminus B_{2^{-j}} , \,\,\, P_{j-1}\cup P_{j} \cup \,P_{j+1}  = B_{2^{2-j} }\setminus B_{2^{-1-j}} \,\,\,\text{and} \,\,\, f_j(2) =\int_{P_{j,\,2}} |\nabla u_i|^n \,dv. $$

 Applying Lemma \ref{Osc} on $P_j$, we have
 \begin{align}\label{estimate_1}
\|u_i\|_{osc(P_j)}&\leq C\,\left(\int_{P_{j-1} \cup P_{j} \cup P_{j+1}}  |\nabla   u_i |^{n}\, dv\right)^\frac{1}{2(n-1)}    \\
& + (2^{-j})^{\frac n{2(n-1)}}\,C \left(\int_{B_{2^{2-j}}} |\tau(u_i)|^2 \, dv\right)^\frac{1}{2(n-1)} \nonumber\\
&\leq C(f_j^{\frac 1{2(n-1)}}(2) + e^{- \lambda_n \, j} ).\nonumber
\end{align}

For $j\geq L_j$, under the assumption \eqref{neck-energy-control} at the beginning of the proof, we can choose a small $\delta$ such that $f_j(L_j)\leq \varepsilon^{2(n-1)}$ and (\ref{f_j(2)_estimate}) yields that
\begin{align}\label{estimate_2}
     f^{\frac 1{2(n-1)}}_j(2)&\leq C \left( e^{\frac{-L_j}{C}}\,f_j(L_j)  + e^{-\lambda_n\,j}\,  e^{\left(\lambda_n-\frac{1}{C}\right)\,L_j} \right)^{\frac 1{2(n-1)}}\\
     &\leq C \left( e^{\frac{-L_j}{C }}\,\varepsilon  +   e^{\frac{-j}{C}} \right).\nonumber
\end{align}
Substituting \eqref{estimate_2} into \eqref{estimate_1} and summing over $j_0\leq j\leq j_n$, we have
\begin{align*}
  \|u_i\|_{osc(B_{2\delta}\setminus B_{2\, r_n R})}&\leq \sum^{j_n}_{j=j_0}\|u_i\|_{osc(P_j)}\\
  &\leq C\sum^{j_n}_{j=j_0}\left(( e^{-\frac{L_j}{C}} \varepsilon + e^{\frac{-j}{C}}) + e^{- \lambda_n\, j}\right)\\
  &\leq C\left( \sum^{\infty}_{i=0}e^{-\frac{i}{C}} \varepsilon + \sum^\infty_{j=j_0} e^{\frac{-j}{C}}\right)\\
  &\leq C\left(\varepsilon + \delta ^\frac{1}{C} \right).
\end{align*}
Since
$$ \|u_i\|_{osc ( B_\delta \setminus B_{ 2 r_nR})} = \sup _{x,y\in B_\delta \setminus B_{2 r_nR}} | u_i(x) - u_i(y)|$$
is controlled by $\delta$, this implies that
$$ u(B_1) \,\cup\,  \omega_1(\mathbb{R}^n)$$ is a connected set. Thus, there is no neck between the limiting map and the bubbles for regular approximated $n$-harmonic maps with tension fields bounded in $L^2$.

\end{proof}

Now we complete the proof of Theorem \ref{no_neck}.

\begin{proof}[\textbf{Proof of Theorem \ref{no_neck}}]
We briefly describe the procedure of ``bubble blowing'' by following the idea from Ding-Tian \cite{DT}. First, we recall that the removable singularity theorem of $n$-harmonic maps \cite{DF}. Moreover, recall the the gap theorem: there is a constant $\varepsilon_{g}>0$ such that if $u$ is an $n$-harmonic map on $S^n$ satisfying $\int_{S^n}|\nabla u|^n<\varepsilon_{g}$, then $u$ is a constant on $S^n$.

Let $u(x,t)$ be a regular solution of the $n$-harmonic flow in $M\times [0,\infty)$.
As $t_i \to \infty$,  it was showed in Proposition \ref{Theorem 3} that a subsequence of  $u_i:=u(t_i)$
converges  to an $n$-harmonic map $u_{\infty}$ locally in $C^{1,\a}(M\backslash\{x^1,\cdots, x^L\} )$. Furthermore, there is a constant $\varepsilon_0>0$ such that  the singular points (energy concentration points) $\{x^k\}$ are defined by the condition
\[ \limsup_{t_i \to \infty}  E_n(u(t_i); B_R(x^k)) \geq \varepsilon_0 \] for any $R\in (0, R_0]$, with some fixed $R_0>0$.

Let $x^1$ be a singular point. Then we find sequences $x^1_i \to x^1$  such that
\[|\nabla u(x^1_i)|=\max_{B_{R_0}(x^1)} | \nabla u(x,t_i)|,\quad  r_i^1 =\frac 1 {|\nabla u(x^1_i)|}\to 0. \]
In  the neighborhood $B_{R_0}(x^1)$ of the singularity $x^1$,   we define the rescaled map

\begin{equation}
	  u^1_i(x)= u(x^1_i+r^1_ix, t_i).
\end{equation}
Then the rescaled map $u^1_i$ satisfies
\begin{equation}\label{Re-n-flow} (r^1_i)^n\frac {\partial  u}{\partial t}=\frac{1}{\sqrt{|g|}}\frac{\partial}{\partial x_{i}}\left[  |\nabla
  u|^{n-2}g^{ij}\sqrt{|g|}\frac{\partial u}{\partial x_{j}}\right]
+|\nabla u|^{n-2}A(u)(\nabla  u,  u).
\end{equation}
Now, $u^1_i$ converge to  $ u_{1,\infty}$ locally in $\R^n$  as $i\to \infty$, and $u_{1,\infty}$ can be extended to a nontrivial $n$-harmonic map on $S^n$ (see \cite {DF}). We call  $\tilde u_{1,\infty}$ to be the first bubble, which satisfies
\begin{equation}\label{First} E_n(u_{1, \infty}; \R^n) =\lim_{R \to \infty} \lim_{t_i \to \infty} E_n(u^1_i; B_{R}(0))=
\lim_{R \to \infty}\lim_{t_i \to \infty}E_n(u_i; B_{Rr^1_i}(x^1)). \end{equation}

  At each singular point $x^k$,  there are  finitely many blow-up points $x_{i}^{k,l}$  and  bubbles $\{\omega_{k,l}\}_{l=1}^{J_k}$ on $\R^n$ (see details in \cite{Hong}); i.e.
at each  $k$, there are sequences $x_{i}^{k,l}\to p^{k,l}$ for some $p^{k,l}$and $r_{i}^{k,l}\to 0$ with $\lim_{i\to \infty}\frac {r_{i}^{k,l}}{r_i^{
k,l-1}}=\infty$ such that passing to a subsequence,
$u_{i}^{k,l}(x):=u_{i} (x^{k,l}_{i}+ r_{i}^{k,l}  x)$ converges to  $\omega^{k,l}$, where $\omega_{k,l}$  is an $n$-harmonic map in   $\R^n$. These mean  that there are  finite   numbers $r_{i,k}$, finite points $x^{k,l}_i$, positive  constants $R_{k,l}$, $\delta_{k,l}$ and  finitely many number  of non-trivial $n$-harmonic
maps
$\omega_{k,l}$ on $\R^n$ such that

\begin{align}\label{neckenergy}
& \lim_{i\to \infty} E_n(u_i; M) \\
=&E_n\left (u_{\infty}; M\backslash \{x_k\}_{k=1}^L\right )+\sum_{k=1}^L\sum_{l=1}^{J_k}
E_n(\omega_{k,l}; \R^n)\nonumber\\
+ &\sum_{k=1}^L \sum_{l=1}^{J_k} \lim_{R_{k,l} \to \infty} \lim_{\delta_{k,l} \to 0} \lim_{i \to \infty} E_n( u_{i}^{k,l}; B_{\delta_{k,l}} \backslash B_{R_{k,l}r_{i}^{
k,l}}(x^{k,l}_i)).\nonumber
\end{align}
Moreover, at each neck region $B_{\delta_{k,l}} \backslash B_{R_{k,l}r_{i}^{k,l}}(x^{k,l}_i)$ in (\ref{neckenergy}), for all $i$ sufficiently large, we have
 \begin{equation}\label{Basic} \int_{B_{2r} \backslash
B_r (x^{k,l}_{i})}{|\nabla  u_{k,l,i}|^n dv} \leq \varepsilon \end{equation}
 for all $r \in (\frac{R_{k,l}r^{k,l}_i}{4},2\delta_{k,l})$, where   $\varepsilon$ is a  fixed constant to be chosen sufficiently small.  In fact, \eqref{Basic} is  a crucial observation by Ding and Tian \cite{DT}. This implies that the neck energy can be controlled during bubbling procedure by reducing multiple bubbles to a
single bubble case, which leads to the proof of the energy identity for harmonic maps in \cite{DT}. For the case of $n$-harmonic maps, we complete the proof of the energy identity by   using a result of Wang and Wei (See Theorem B of \cite{WW}) .
Now, we can choose a subsequence of time $t_i\to \infty $ such that  $\lim_{i\to \infty}\left\|\frac {\partial u}{\partial t}(\cdot ,t_i)\right\|_{L^2(M)} $
is bounded. This completes a proof of Theorem 1 by using Theorem 3.
\end{proof}

\section{Finite-time blow-up of the $n$-harmonic map flow}
\label{sec:main_result}
As an application of the ``no-neck'' result, we will  construct an example  that the $n$-harmonic flow with initial value $u_0$ blows up in finite time. The proof here is to use  similar ideas  in  \cite{LY}. Due to that there are several modifications for the case of $n$-harmonic maps,   we give a proof for completeness here.

\subsection{Width of $n$-harmonic maps in the covering space}
We follow the geometric setting as in Sections 3-4 of \cite{LY} to construct an example  of finite time blowup of the $n$-harmonic flow. The idea is to construct a proper target manifold $N$ such that we can find infinitely many  initial maps $u_0:S^n\to N$ such that the $n$-harmonic flow   blows up in finite time.

For $m > n$, let the target manifold $N= X\,\#\,T^m$   be the connected sum of $X$  with the torus
$T^m$. Here $X$ is a  closed $m$-dimensional manifold with
nontrivial $\pi_n(X)$. Thus, there exists a smooth map $h: S^n \to X $ such that it is not homotopic to a constant map.
Note that $N$ can be separated into $N_1$ and $N_2$ by an embedding sphere $S^{m-1}\subset N$. In particular, $N\setminus N_1$ and $N\setminus N_2$ are homeomorphic to $X$ and $T^m$ respectively.
For each $l=0, 1,2,...$, let $U_l$ denote a small neighborhood of $p_l$, which is diffeomorphic to a $m$-dimensional ball and $V\subset X$ denotes an open set which is diffeomorphic to a ball.

$\R^m$ is the universal cover of $T^m$ with $G=\mathbb{Z}^m$ as the covering transformations group. Now, for any point $p_0\in \mathbb{R}^m$,  its orbit under the transformation  group $G$ is the set $\{p_l\}_{l=0}^\infty\subset \R^m$. Let $U_0$ be a small ball in $\R^m$ and its orbit under the transformation  group $G$ is a family of balls $\{U_l\}_{l=0}^\infty\subset \R^m$. Now,  we can find a cover of $N$ by modifying $\R^m$. For each $l=0,\cdots, \infty$, we remove the  small ball $U_l$ from $\mathbb{R}^m$ for $l=1,2,...$ by adding a copy of $X\setminus V$, which we identify $\partial U_l$ by the  boundary of $X\setminus V$. We  denote by $X_l$ the copy of $X\setminus V$ through $\partial U_l$. This new complete and non-compact manifold is   denoted by $\tilde{N}$ and the transformation  group $G$ act naturally on $\tilde{N}$. Let $\tilde{N}$ to be a cover of $N$ and $\tilde{g}$ be the corresponding lift metric.

For a continuous map $u: S^n\to N$,  we define its    ``width'' of $u$ in a set $S\subset S^{n-1}$ through its lift map $\tilde u$ in the covering space $(\tilde{N}, \tilde{g})$  by

\begin{equation}\label{width_of_the_map}
	\mathbf{\mathcal{W}}(u; S):= \sup_{x,y\in S}\, d_{(\tilde{N},\,\tilde{g})} \left(\, \tilde{u}(x), \,  \tilde{u}(y) \,\right).
\end{equation}

We begin with the following lemma that gives an upper bound for the width.
\begin{lem}(Bounded width lemma)\label{bounded_width}
If $u: \mathbb{R}^n \to N$ is an $n$-harmonic map with $E_n(u)< C_1$ for a constant $C_1>0$, there exists a constant $C_2$, depending only on $C_1$ and $N$  such that the $\mathcal{W}(u; S^{n})< C_2$.
\end{lem}

\begin{proof} We prove this by contradiction. Suppose that the statement is not true. Then we can   find a sequence of $n$-harmonic maps $\{u_i\}_{i=1}^{\infty}$ with their energy  bounded by the constant $C_1$ such that their width $\mathcal{W}(u_i; S^{n})$ can not be bounded as $i\to \infty$.

According the above bubble-neck decomposition, as $t_i \to \infty$,  it was showed in Proposition \ref{Theorem 3} that a subsequence of  $u_i$ converges  to an $n$-harmonic map $u_{\infty}$ locally in $C^{1,\a}(M\backslash\{x^1,\cdots, x^L\} )$.

At each singular point  $x_k$, there are sequences $x_{i}^{k,l}\to p^{k,l}$ for some $p^{k,l}$and $r_{i}^{k,l}\to 0$ with $\lim_{i\to \infty}\frac {r_{i}^{k,l}}{r_i^{
k,l-1}}=\infty$ such that passing to a subsequence,
$u_{i}^{k,l}(x):=u_{i} (x^{k,l}_{i}+ r_{i}^{k,l}  x)$ converges to  $\omega^{k,l}$, where $\omega_{k,l}$  is an $n$-harmonic map in   $\R^n$. These mean  that there are  finite   numbers $r_{i,k}$, finite points $x^{k,l}_i$, positive  constants $R_{k,l}$, $\delta_{k,l}$ and  finitely many number  of non-trivial $n$-harmonic
maps
$\omega_{k,l}$ on $\R^n$.  Moreover, at each neck region $B_{\delta_{k,l}} \backslash B_{R_{k,l}r_{i}^{k,l}}(x^{k,l}_i)$ in (\ref{neckenergy}), for all $i$ sufficiently large, we have
 \begin{equation}\label{Basic1} \int_{B_{2r} \backslash
B_r (x^{k,l}_{i})}{|\nabla  u_{k,l,i}|^n dv} \leq \varepsilon \end{equation}
 for all $r \in (\frac{R_{k,l}r^{k,l}_i}{4},2\delta_{k,l})$, where   $\varepsilon$ is a  fixed constant to be chosen sufficiently small.
Then
\begin{align}\label{neckenergy1}
& \lim_{i\to \infty} W(u_i; S^n) =\lim_{i\to \infty} \sup_{x,y\in S^n}\, d_{(\tilde{N},\,\tilde{g})} \left(\, \tilde{u_i }(x),\, \tilde{u_i}(y)\,\right )\\
\leq &\lim_{\delta \to 0}\lim_{i\to \infty} W\left (u_{i}; S^n\backslash \cup ^L _{k=1} B_\delta (x_k)\right )\nonumber\\
&+ \lim_{\delta_{k,l} \to 0}\lim_{i\to \infty}\sum_{k=1}^L\sum_{\tilde l=1}^{\tilde J_k}
W(u_{i}^{k,\tilde l}; B_{R_{k,\tilde l}}(0)\backslash  \cup_{j=1}^{L_{k,\tilde l}} B_{\delta_{k,l}}(x^{k,j}_i)  )\nonumber\\
&+ \sum_{k=1}^L \sum_{l=1}^{J_k} \lim_{R_{k,l} \to \infty} \lim_{\delta_{k,l} \to 0} \lim_{i \to \infty} W( u_{i}^{k,l}; B_{\delta_{k,l}} \backslash B_{R_{k,l}r_{i}^{
k,l}}(x^{k,l}_i)),\nonumber
\end{align}
where we note that $\{x_i^{k,l}\}_{l=1}^{J_k}=\cup_{\tilde l=1}^{\tilde J_k}\{x_i^{k,j}\}_{j=1}^{L_{k,\tilde l}}$ is the set of totally blowing points and that $L_{k,\tilde l}$ may not exist (This corresponds to the case of a single bubble).

Now we will estimate the width of the above region of bubbling, the neck domain and the base separately.
 Let $\tilde u_i^{k,l}$ denote the lift of $u_i^{k,l}$. Since $u_{i}^{k,l} \to \omega_{k,l}$ locally in $C^{1,\alpha }(\R^n\backslash\{p_{k,j}\}_{j=1}^{J_l} )$,  the lift $\tilde u_{i}^{k,l}$    convergence  to the lift $\tilde \omega_{k,l}$ in the covering space with lift metric $\tilde{g}$ as well, so
$$  \lim_{\delta_{k,l} \to 0}\lim _{i\to \infty}\sup_{x\in \R^n\backslash  \cup_{l=1}^{J_k} B_{\delta_{k,l}}(x^{k,l}_i))\}}  d_{(\tilde{N},\,\tilde{g})} \left(\, \tilde u_i^{k,l},\, \tilde\omega_{k,l}(x)\,\right) = 0.$$
By the triangle inequality, we have
$$  \lim_{\delta_{k,l} \to 0}\lim_{i\to \infty}\sum_{k=1}^L\sum_{l=1}^{J_k}
W(u_{i}^{k,l}; \R^n\backslash \cup_{l=1}^{J_k} B_{\delta_{k,l}}(x^{k,l}_i))\leq \sum_{k=1}^L\sum_{l=1}^{J_k}\mathcal{W}(\omega_{k,l}; \R^n).$$

Similarly,   we have
$$\limsup_{i\to\infty} \sup_{x,y\in \mathbb{R}^n\backslash \cup ^l _{k=1} B_\delta (x_k)}\, d_{(\tilde{N},\,\tilde{g})}\, \left(\tilde{u_i}(x),\,  \tilde{u_i}( y)\right) \leq \mathcal{W} (u_\infty).$$

By the no-neck result in Theorem \ref{no_neck}, we have
$$\lim_{R_{k,l} \to \infty} \lim_{\delta_{k,l} \to 0} \lim_{i \to \infty} W( u_{i}^{k,l}; B_{\delta_{k,l}} \backslash B_{R_{k,l}r_{i}^{
k,l}}(x^{k,l}_i))=0$$

These imply that
\begin{align}
\lim_{i\to \infty} W(u_i; S^n)
\leq &W\left (u_{\infty}; S^n\right )+\sum_{k=1}^L\sum_{l=1}^{J_k}
W(\omega_{k,l}; \R^n),\nonumber
\end{align}
 which is contradicted with the assumption. This proves the claim.
\end{proof}

As a consequence, we have
\begin{lem} \label{SE_width}Let $u$ be a regular solution to (\ref{n-flow}) in $M\times [0,\infty)$ with initial value $u_0$ satisfying $E_n(u_0)< C_1$ for a constant $C_1>0$. Then there is a sequence $t_i\to\infty$ such that
$u(\cdot , t_i)$ converges to an $n$-harmonic maps $u_{\infty}$  in $C_{loc}^{1,\a}(M\backslash \{x^1,...,x^L\})$.
Moreover, there exists a constant $C_3$, depending only on $C_1$,  such that the
\[\limsup_{i\to\infty}\mathcal{W}(u(\cdot , t_i); S^n)\leq  C_3.\]
\end{lem}
\begin{proof}
By using Theorem 1,  there exists a sequence $t_i\to \infty$ such that  $u(t_i)$ converges to an $n$-harmonic map $u_{\infty}$  in $C^{1,\a} (M\backslash \{x^1, \cdots, x^L\})$ for some positive $\a<1$.
Moreover, there are a finite number of
$n$-harmonic maps $\omega_{k,l}$ on $S^{n}$ with $k=1,...,L$ and $l=1,...,J_k$. By applying   Lemma 4.1, we have
\begin{align*}
\limsup_{i\to\infty}\mathcal{W}(u_i; S^n)&\leq  \mathcal{W}(u_{\infty}; S^n)+\sum_{k=1}^L\sum_{l=1}^{J_k} \mathcal{W}( \omega_{k,l}; S^n)\leq C_3,
\end{align*}
where $C_3$ depends on $C_2$ and total numbers of bubbles.
\end{proof}

With this bounded width lemma, we are now ready to construct the example of  the $n$-harmonic map flow with initial map $u_0: S^n\to N $ which blows up in finite time. The basic idea is as follows: We construct an initial $u_0:S^n \to N$ which has finite energy. Then we see if a map $u'$, which is homotopic to $u_0$, could have a large width.

\subsection{Proof of Theorem \ref{main_theorem}} Since  $X$ is a  closed manifold of dimension $m>n$ with
nontrivial $\pi_n(X)$, we can
find a smooth map $ h:S^n\to X$ such that
\begin{enumerate}[label=(\alph*)]
	\item $h$ is non-subjective;
	\item $h$ is not homotopic to any constant map;
	\item $h(S^n)\subset X\setminus \overline{V}$;
	\item $h$ maps the southern hemisphere of $S^n$ to a point $q\in X\setminus V$.
\end{enumerate}
For each $l=0, 1, ...$, we denote $h_l: S^n \to X_l\subset \tilde{N}$ as a copy of $h$ and $q_l\in X_l$ as a copy of $q$, and also denote by $\mathcal{S}_p$ the south pole of $S^n$.

For any large constant $K>0$, there is a sufficiently large $l$ such that
\[d_{(\tilde{N},\,\tilde{g})} \left(X_0,\, X_l\,\right)\geq K.\]
Let $q_0\in X_0$ and $q_l\in X_l$ be copies of $q$. Let $\Psi$ and $\Phi$ be two  stereographic projections from $S^n$ to $\R^n$ given by
\begin{align} &\Phi (x^1,\cdots, x^n, x^{n+1})=\left (\frac {x^1}{1-x^{n+1}}, \cdots, \frac {x^n}{1-x^{n+1}} \right ),\\
&\Psi (x^1,\cdots, x^n, x^{n+1})=\left (\frac {x^1}{1+x^{n+1}}, \cdots, \frac {x^n}{1+x^{n+1}} \right ),
\end{align}
which map the north pole $\mathcal{N}_p$ and the south pole $\mathcal{S}_p$ of $S^n$ to the infinity respectively.

In order to   construct an initial map $u_0: S^n\to N$, we  define  a map $\tilde u_0: S^n \to \tilde{N}$  by
\begin{equation}
    \tilde u_0 = \begin{cases}

               h_0(x) ,   & \text{for}\, x \in S^n\setminus B_\sigma(\mathcal{S}_p);\\
               \gamma \, \circ \, \varphi\left(\frac {\log \sigma - \log |x|  }{-\log \sigma} \right), &\text{for}\, x\in B_\sigma (\mathcal{S}_p)\setminus B_{\sigma^2}(\mathcal{S}_p);\\
               h_l\circ \Phi^{-1} \circ(\frac{\Psi (x)}{\sigma^2/ 2})  & \text{for}\, x\in B_{\sigma^2}(\mathcal{S}_p).\\
\end{cases}
\end{equation}
Here $\gamma:[0,1]\to \tilde{N}$ is the shortest geodesics connecting $q_0$ to $q_l$ in $\tilde{N}$, and $\varphi$ is a smooth cut-off function on $[0,1]$ that satisfies:
\begin{enumerate}
  \item $\varphi'$ is non-negative and $|\varphi |\leq 1$;
  \item $\varphi(x) = 0$, for $x\in \left[0,\frac{1}{8}\right]$ and $\varphi(x)=1$ for $x\in \left[\frac{7}{8}, 1\right];$
  \item $|\varphi'|\leq C$, where $C$ is a constant.
\end{enumerate}
 Under the definition of $\tilde u_0$, we can see  $\tilde u_0|_{\partial B_{\sigma} ( \mathcal{S}_p)} = q_0$. Moreover, for small $\sigma$, the metric was flattened out which gives $\tilde{u_0}|_{\partial B_{\sigma^2}(\mathcal{S}_p)} = q_l$.

Given that we have $\tilde{g}$ as the pullback metric for the covering of $(N,g)$, there exists an isometric projection map $\pi :\tilde{N}\to N $. For sufficiently small $\sigma$, we can find a smooth $u_0: S^n\to N$  defined by
\begin{equation}
    u_0 = \begin{cases}

               \pi\,\circ h_0(x),   & \text{for}\, x \in S^n\setminus B_\sigma(\mathcal{S}_p);\\
               \pi\, \circ \,\gamma \, \circ \, \varphi\left(\frac {\log \sigma - \log |x|  }{-\log \sigma} \right), &\text{for}\, x\in B_\sigma (\mathcal{S}_p)\setminus B_{\sigma^2}(\mathcal{S}_p);\\
               \pi\,\circ h_l\circ \Phi^{-1} \circ(\frac{\Psi (x)}{\sigma^2/ 2}),  & \text{for}\, x\in B_{\sigma^2}(\mathcal{S}_p).\\
\end{cases}
\end{equation}
Now we claim that there is a constant $C_1$ depending on $h_0$ such that
\begin{equation}\label{total-energy-bound}
	E_n(u_0)< E_n(h_l)+ E_n(h_0) +1 = C_1.
\end{equation}
\\
Due to the fact that $E_n(u)$ is conformally invariant,  the energy $E_n(u_0)$ over $S^n\setminus B_\sigma(\mathcal{S}_p)$ and $B_{\sigma^2}(\mathcal{S}_p)$ for small $\sigma$ can be bounded by

\begin{align}\label{energy_bound_on_the_gap}
\int_{S^n\setminus B_\sigma(\mathcal{S}_p)} |\nabla u_0| ^n \, dv +\int_{ B_{\sigma^2}(\mathcal{S}_p)} |\nabla u_0| ^n \, dv
&\leq E_n(h_0) + E_n(h_l) + \frac{1}{2}.
\end{align}

Now we have to check if the condition \eqref{total-energy-bound} is satisfied. We do this by estimating the energy over $B_\sigma (\mathcal{S}_p)\setminus B_{\sigma^2}(\mathcal{S}_p)$, then compare it with \eqref{energy_bound_on_the_gap}.

Let $L$ be the shortest distance between $q_0$ and $q_l$. Since $\gamma$ is the shortest geodesics connecting $q_0$ to $q_l$ in $\tilde{N}$,
there is a parametrization $\tilde s$ such that
\begin{align*}
 \int_0^1 |(\pi\,\circ\,\gamma)'| \,d\tilde s = d_{(\tilde{N},\tilde{g})} ( q_0, q_l)=L,\quad  |(\pi\,\circ\,\gamma)'|  =  L.
\end{align*}

Therefore, we have

\begin{align*}
  |\partial_r\, u_0| &\leq |(\pi\,\circ\,\gamma)' | \,| \varphi' | \,  \frac{1}{r(-\log \sigma)}
  \leq \frac{CL}{r(-\log\sigma)}
\end{align*}
which gives us to estimate the energy of $u_0$ on the  annulus  domain; i.e.
\begin{align*}
  &\int_{B_\sigma\setminus B_{\sigma^2}} |\nabla u_0|^n \, dx \leq C \int_{\sigma^2}^{\sigma}  |\partial_r\, u_0| ^n \,r^{n-1}\, dr\\
  &\leq \frac{CL^n}{(-\log \sigma)^n} \int_{\sigma^2}^{\sigma} \frac{1}{r}\,dr\leq \frac{CL^n}{(-\log \sigma)^{n-1}}.
\end{align*}
Therefore, the energy on the  annulus domain $B_\sigma\setminus B_{\sigma^2}$ can be controlled for any $L$ with a sufficiently small $\sigma$. Together with \eqref{energy_bound_on_the_gap}, we obtained an upper bound $C_1$ for $E(u_0)$.

Now, for any $u'$ (with a lift $\tilde u'$) which is homotopic to $u_0$, we claim that $\tilde u'$ intercepts with $X_0$ and $X_l$, which implies
\[\mathbf{\mathcal{W}}(u'; S^n)\geq d_{(\tilde{N},\,\tilde{g})}   \left(X_0,\, X_l\,\right)\geq K > C_3.\]
We prove this claim by contradiction. Assume that $\tilde u'$ does not intercept  with $X_0$.  Set a continuous map $\overline{\pi}: \tilde{N} \to X$ so that $\overline{\pi}$ maps $\tilde{N}\setminus X_0$ to a single point $p\in X$. Since  $u'\cap X_0 =\emptyset$, it follows that  $\overline{\pi} \,\circ\, \tilde{u'} $ maps to $p$ which is a constant map. However, consider $\overline{\pi} \,\circ\, \tilde u'$ is homotopic to $\overline{\pi} \,\circ\, \tilde u_0$ which is homotopic to $h_0$ as well. This contradicts with the property (b) of the  definition of $h_0$. This shows that $\tilde u'$ must intercept  with $X_0$.
By a similar argument,  $\tilde u'$ must intercept  with $X_l$.

Assume that the $n$-harmonic map flow with initial value $u_0$ does not blow  up in finite time.  Let $u(x,t)$ be a  regular solution to the flow (\ref {n-flow}) in $M\times [0,\infty)$ with initial value $u_0\in C^1(M, N)$. By Theorem 1, there is a sub-sequence  $t_i$  such that as $t_i\to \infty $, $u(x, t_i)$ converges to an $n$-harmonic map $u_{\infty}$  in $C_{loc}^{1,\a} (M\backslash \{x^1, \cdots, x^L\})$ for some positive $\a<1$.
Since $u_i:=u(x, t_i)$ is  homotopic to $u_0$, we have
\[\mathbf{\mathcal{W}}(u_i; S^n)\geq  K > C_3.\]
On the other hand, by Lemma \ref{SE_width}, $\limsup_{i\to \infty}\mathbf{\mathcal{W}}(u_i; S^n)\leq C_3$. This is a contradiction.
Therefore, we have constructed an initial maps $u_0 : S^n \to N$ such that the $n$-harmonic map flow with initial value $u_0$ must blow  up in finite time. This completes a proof of Theorem 2.\qed

\medskip \noindent{\bf Remark.}
{\it It is an interesting question whether the  the heat flow for
$H$-systems (\cite {HH}) on $n$-manifolds  blows up in finite time for $n\geq 3$.} \medskip

\begin{acknowledgement}{The first author was supported by a top-up PhD scholarship in the Australian Research Council
grant DP150101275. The research  of the second
author was supported by the Australian Research Council
grant DP150101275. }
\end{acknowledgement}

\today


\begin{thebibliography}{}

\bibitem{CDY}
    K.C. Chang, W.Y. Ding and R. Ye,
    \newblock{Finite-time blow-up of the heat flow of harmonic maps from surfaces},
    \newblock{\em J. Differential Geom.}, {\bf 36} (1992) 507--515.

    \bibitem{CCCL}
     C.-N. Chen, L. F. Cheung, Y. S. Choi, and C. K. Law,
    \newblock{On the blow-up of heat flow for conformal $3$-harmonic maps},
    \newblock{\em Trans. Amer. Math. Soc.}, {\bf 354} (2002)  5087--5110.

\bibitem{CL}
    J. Chen and Y. Li,
    \newblock{Homotopy classes of harmonic maps of the stratified $2$-spheres and applications to geometric flows},
    \newblock{\em Advances in Mathematics}, {\bf 263}(2014)357--388.


\bibitem{CG}
J-M. Coron and J-M. Ghidaglia.
\newblock Explosion en temps fini pour le flot des applications harmoniques.
\newblock {\em C. R. Acad. Sci. Paris S\'er. I Math.}, {\bf 308} (1989) 339--344.

\bibitem{DF1}
     E. DiBenedetto   and A. Friedman,
    \newblock{Holder estimates for nonlinear degenerate
parabolic Systems},
    \newblock{\em J. Reine Angew. Math.}  {\bf 357}  (1985)  1--22.

\bibitem{DT}
    W. Ding and G. Tian,
    \newblock{Energy identity for a class of approximate harmonic maps from
  surfaces.},
    \newblock{\em Comm. Anal. Geom.} {\bf 3} (1995) 543--554.



\bibitem{DF}  F. Duzaar and M. Fuchs, \newblock {On removable singularities of p-harmonic maps,} \newblock {\em Annales de l'IHP Analyse non linéaire}.{\bf 5} (1995) 543--554.

\bibitem{EL}
        J. Eells and L. Lemaire,
        \newblock{A report on harmonic maps,}
        \newblock{\em Bull. London Math. Soc.}, {\bf 10,1} (1978) 1--68.

\bibitem{ES}
        J. Eells and J.H. Sampson,
        \newblock{Harmonic mappings of Riemannian manifolds,}
        \newblock{\em Amer. J. Math.}, {\bf 86} (1964) 109--160.



\bibitem{GT}
    D. Gilbarg and N. Trudinger,
        \newblock{Elliptic partial differential equations of second order,}
        \newblock{Second edition. Springer-Verlag, Berlin, 1983.}

 \bibitem{Ha}
    R. Hamilton,
        \newblock{Harmonic maps of manifolds with boundary,}
        \newblock{ {\em Lecture Notes in Mathematics}, Vol. 471. Springer-Verlag, Berlin-New York, 1975.}


\bibitem{H}
F. H{\'e}lein.
\newblock R\'egularit\'e des applications faiblement harmoniques entre une
  surface et une vari\'et\'e riemannienne.
\newblock {\em C. R. Acad. Sci. Paris S\'er. I Math.}, 312(8):591--596, 1991.


\bibitem{Hong}
M.-C. Hong,
\newblock The rectified n-harmonic map flow with applications to homotopy classes.
 \newblock {\em To appear in Ann. Scuola Norm. Sup.
Pisa Cl. Sci.}


\bibitem  {HH} M.-C. Hong  and D. Hsi,  \newblock{The heat flow for
$H$-systems on higher dimensional manifolds},  \newblock{\em Indiana Univ. Math.
J.} {\bf 59} (2010) 761--790.

\bibitem{HY}
  M.-C. Hong   and H. Yin,
    \newblock{On the  Sacks-Uhlenbeck flow of Riemannian surfaces},
  \newblock  {\em Communications in Analysis and Geometry} {\bf 21} (2013) 917--955.


\bibitem{Hungerb1994}
N. Hungerb\"{u}hler,  \newblock {\em $p$-harmonic flow},    \newblock {{PhD} thesis, ETH Z{\"u}rich, Diss. Math. Wiss, 1994.}





\bibitem {Hung}N. Hungerb\"{u}hler,  \newblock {$m$-harmonic flow},    \newblock {\em Ann. Scuola Norm. Sup.
Pisa Cl. Sci.} (4) XXIV, {\bf 4} (1997), 593--631.

\bibitem  {J} J. Jost,  \newblock{Harmonic maps between surfaces},  \newblock{Springer-Verlag, Berlin} (1984)



\bibitem{LZ}
  J. Li   and X. Zhu,
    \newblock{Energy identity for the maps from a surface with tension field bounded in $L^p$},
  \newblock  {\em Pacific Journal of Mathematics} {\bf 260-1} (2012), 181-195.



\bibitem {LY} L. Liu and H. Yin,  \newblock {On the finite time below-up  of  biharmonic map
flow in dimension four}, J. Elliptic Parabol. Equ. 1 (2015), 363-385.

\bibitem{LY2}
L. Liu and H. Yin,
\newblock Neck analysis for biharmonic maps.
\newblock {\em Math. Z.}, 283(3-4):807--834, 2016.


\bibitem{P} 	T. Parker, 	\newblock{Bubble tree convergence for harmonic maps}, 	\newblock{\em J. Differential Geom.}, {\bf 44} (1996), 595--633.

\bibitem{Q} J. Qing, \newblock{On singularities of the heat flow for harmonic maps from surfaces
  into spheres},  \newblock{\em Comm. Anal. Geom.}, 3(1-2):297--315, 1995.


\bibitem{QT} J. Qing and G. Tian, \newblock {Bubbling of the heat flows for harmonic maps from surfaces,} \newblock   {\em Comm. Pure Appl. Math.} {\bf 50} (1997) 295--310.

\bibitem{SU}
         J. Sacks and K. Uhlenbeck,
        \newblock{The existence of minimal immersions of $2$-spheres},
    \newblock{\em Ann. of Math.}, {\bf 113} (1981), 1--24.


\bibitem{St}
  M. Struwe,
    \newblock{On the evolution of harmonic maps of Riemannian surfaces},
    \newblock{\em Commun. Math. Helv.}, {\bf 60} (1985), 558--581.


  \bibitem{WW}
C. Wang and  S. Wei,
    \newblock{Energy identity for $m$-harmonic maps},
    \newblock{\em Differential Integral Equations}, {\bf  15} (2002),  1519--1532.




\end{thebibliography}
\end{document}